\documentclass[letterpaper, 10 pt, conference]{ieeeconf}  

\IEEEoverridecommandlockouts                              

\usepackage{amsmath}
\usepackage{amssymb}
\usepackage{graphicx}
\usepackage{subcaption}
\usepackage{xcolor}
\usepackage{algorithm}
\usepackage{url}
\usepackage{algpseudocode}
\usepackage{cite}

\newtheorem{definition}{Definition}
\newtheorem{theorem}{Theorem}
\newtheorem{lemma}{Lemma}

\newtheorem{proposition}{Proposition}

\title{Verification Framework for the
Union of Control Barrier Functions
}
\author{
Chuanrui Jiang and Andrew Clark
\thanks{Chuanrui Jiang and Andrew Clark are with the Electrical and Systems Engineering Department, McKelvey School of Engineering, Washington University in St. Louis, St. Louis, MO 63130. Email: \{chaunrui,andrewclark\}@wustl.edu}
}

\begin{document}
\maketitle

\begin{abstract}
Control Barrier Functions (CBFs) have been proposed to ensure safety of autonomous systems. This paper considers control policies that switch between CBF constraints. Under this approach, we represent a complex non-convex safe region as a union of sets that are computationally tractable to verify. We denote this framework as union-CBFs and make the following contributions. First, considering switching CBF-QP controllers, we propose a sufficient condition that ensures (i) the system undergoes a finite number of switches in any finite time interval and ensures (ii) the forward invariance of the closed-loop system in between switches. Second, we consider two types of switching strategies and propose union-CBFs conditions for each strategy to satisfy (i) and (ii). Third, we formulate Sum-of-Squares (SOS) algorithms to verify the conditions. The experiments show that our union-CBFs framework results in a larger safe region compared to high-degree polynomial CBFs. We also show the efficiency of the verification algorithms using a polynomial system model. Code: \url{https://github.com/Chuanruijiang/union_cbf_acc_2026.git}
\end{abstract}
\section{Introduction}
\label{sec:intro}
 
Safety is a critical property of control systems and is usually defined as the forward invariance of a set of safe states\cite{franco2008set}. Control Barrier Functions (CBFs) have been proposed to ensure safety \cite{ames2016control, ames2019control}. Mathematically, CBFs are state evaluation functions that output negative values for unsafe states and output positive values for some safe states. The safety of a system can be ensured by placing a CBF constraint on the control input so that the closed-loop system is forward invariant within the 0-super level set of the CBF.

In order to ensure the system perform needed tasks, it is usually desirable to maximize the volume of the safe region \cite{srinivasan2020synthesis, dai2023convex, xiao2023barriernet}. In many safe control scenarios, the set of safe states is non-convex  with irregular shapes \cite{glotfelter2017nonsmooth, molnar2023composing}. In such cases, the CBF could be a high-degree polynomial \cite{wang2024safe} or neural network \cite{jin2020neural, robey2020learning, dawson2023safe}, both of which require significant computational overhead to synthesize. Furthermore, verifying the safety properties of such complex CBFs is computationally prohibitive \cite{zhang2023exact}.

An alternative to expressing complex safety constraints with a single CBF is to describe the safe region as a union of simpler subsets, with each subset defined by a distinct CBF \cite{vahs2024non, glotfelter2020nonsmooth}. We call this idea \textit{union-CBFs}. 
One methodology \cite{molnar2023composing} composes these distinct CBFs to form a single CBF. The resulting composed CBF then ensures safety in a subset of the union of each CBF's safe region. Another non-smooth CBF methodology \cite{glotfelter2020nonsmooth, vahs2024non} partitions the union region into disjoint subsets. Safety is then ensured by switching between CBFs when system states enter different subsets. However, at states where subsets connect, the safe controller is required to satisfy multiple CBF constraints. This may cause numerical issues due to over constraining. The recent work \cite{long2025safety} proposes necessary and sufficient conditions for safety using switching CBFs, but assumes the uniqueness of trajectories of the closed-loop system. According to \cite{clark2021verification, mestres2025regularity}, formal CBF-based safety also requires non-trivial regularity properties for safe controllers (e.g. existence of a feasible control input \cite{clark2021verification}, and smoothness\cite{morris2015continuity} or local Lipschitz \cite{mestres2025regularity} properties). These regularity properties are not guaranteed for optimization-based controllers, such as CBF-Quadratic Programming\cite{ames2016control} (CBF-QP), and hence must be verified \cite{clark2021verification, zhang2023exact}.

In this paper, the first result is a union-CBFs framework with switching CBF-QP controllers. This framework overcomes the over constraining issue, and ensures (i) the forward invariance of the same union safe region as \cite{glotfelter2020nonsmooth, vahs2024non}, and (ii) the number of switches in any finite time interval is finite. The switching CBF-QP within this framework satisfies needed regularity properties. Based on the first result, we propose two types of switching strategies. For each strategy, we derive forward invariant conditions and develop Sum-Of-Squares-based verification algorithms to verify these conditions. In the experiments, we show that the proposed union-CBFs framework is (i) more efficient compared to high-degree polynomials and composed CBF and (ii) avoids the numerical issue brought by the over constraining. Efficiency of our proposed SOS verification methods are also analyzed using a nonlinear polynomial system example.

The rest of the paper is organized as follows. We introduce preliminaries in Section \ref{sec:preliminaries}. We propose sufficient conditions for forward invariance under union-CBFs in Section \ref{sec:fi_condition}. The union-CBFs conditions under two switching strategies are introduced in Section \ref{sec:strategies}. The SOS verification frameworks are derived in Section \ref{sec:verification}. Experiments are presented in Section \ref{sec:experiments}, and Section \ref{sec:conclusion} concludes the paper.

\textbf{Notations: } Throughout this paper, we use un-bold letters $x$ and $X$ to denote scalars, use bold lower case letters $\mathbf{x}$ to denote vectors, and use bold upper case letters $\mathbf{X}$ to present matrices. The bold $\mathbf{0}$ denotes a vector of all zeros. For a vector $\mathbf{x}$, $\mathbf{x}_{(i)}$ denotes the $i$-th element. $\mathbf{x}^2$ denotes element-wise squared vector of $\mathbf{x}$ such that $(\mathbf{x}^2)_{(i)} = (\mathbf{x}_{(i)})^2, i=1,\dots,n$. For a vector $\mathbf{x}$ and a scalar constant $\epsilon$, $\mathbf{y}=\mathbf{x}-\epsilon$ means $\mathbf{y}_{(i)}=\mathbf{x}_{(i)}-\epsilon$ for all $i=1,\dots,n$. For two vectors $\mathbf{x},\mathbf{y}\in\mathbb{R}^{n}$, $\mathbf{x}\leq\mathbf{y}$ denotes that $\mathbf{x}_{(i)}\leq\mathbf{y}_{(i)}$ for all $i=1,...,n$. For a matrix $\mathbf{X}$, $\mathbf{X}\succ 0$ means $\mathbf{X}$ is positive definite. We also use upper case letters in Cali-graphic font $\mathcal{X}$ to present sets. For Lie derivatives, we let $L_\mathbf{f}h(\mathbf{x}) = \frac{\partial h(\mathbf{x})}{\partial \mathbf{x}}\mathbf{f}(\mathbf{x})$. 
\section{Preliminaries}
\label{sec:preliminaries}

In this section, we first introduce the system model. Then, we introduce the concepts of Control Barrier Functions (CBFs). Finally, we introduce a Sum-Of-Squares method that verifies the emptiness of a set.
\subsection{System Model}
We consider a control affine system defined as
\begin{equation}
\label{eq:sys_dyn_control_aff}
    \Dot{\mathbf{x}}(t) = \mathbf{f}(\mathbf{x}(t)) + \mathbf{G}(\mathbf{x}(t))\mathbf{u}(t)
\end{equation}
where the state $\mathbf{x}(t)\in\mathcal{X}\subseteq\mathbb{R}^{n_x}$, the control input $\mathbf{u}(t)\in\mathcal{U}\subseteq\mathbb{R}^{n_u}$, $\mathcal{X}$ denotes the state space, and $\mathcal{U}$ denotes the admissible control set. In this paper, we assume that both $\mathbf{f}(\mathbf{x})$ and $\mathbf{G}(\mathbf{x})$ are polynomials of $\mathbf{x}$. We also assume that the admissible control set is $\mathcal{U}=\{\mathbf{u}|\mathbf{A}\mathbf{u}\leq\mathbf{c}\}$, where $\mathbf{A}$ and $\mathbf{c}$ are constant matrix and vector such that $\mathcal{U}$ is compact. The state space $\mathcal{X}$ has a safe region $\mathcal{X}_s$ containing all the safe states, and an unsafe region $\mathcal{X}_u = \mathcal{X}\setminus\mathcal{X}_s$ containing all the unsafe states. The system \eqref{eq:sys_dyn_control_aff} is safe if there exists a controller and a subset $\mathcal{C}\subseteq\mathcal{X}_s$ such that, $\forall \mathbf{x}(0) \in \mathcal{C}$, the closed loop system satisfies $\mathbf{x}(t)\in\mathcal{C}, \forall t\geq 0$, namely, system \eqref{eq:sys_dyn_control_aff} is controlled forward invariant with respect to $\mathcal{C}$.

\subsection{Control Barrier Functions}
In order to ensure safety, we define a function $h:\mathcal{X}\rightarrow\mathbb{R}$. We let $\mathcal{C}$ be its 0-super-level set $\{\mathbf{x}|h(\mathbf{x})\geq 0\}$ so that the safety can be certified by control barrier functions.
\begin{definition}[CBF\cite{ames2019control}]
\label{def:cbf}
A continuously differentiable function $h(\mathbf{x})$ is a CBF for system \eqref{eq:sys_dyn_control_aff} if there exists an extended class-$\kappa$ function $\kappa_h( )$ such that for all $\mathbf{x}\in\mathcal{C}$:
\begin{equation}
\label{eq:def_cbf}
    \sup_{\mathbf{u}\in \mathcal{U}}\{L_{\mathbf{f}}h(\mathbf{x})+L_{\mathbf{G}}h(\mathbf{x})\mathbf{u}+\kappa(h(\mathbf{x}))\}\geq 0.
\end{equation}
\end{definition}
Note that $L_{\mathbf{f}}h(\mathbf{x})$ is a scalar, $L_{\mathbf{G}}h(\mathbf{x})$ is a 1-by-$n_u$ row vector, and $\kappa(.)$ is an extended class $\kappa$ function, which is monotonically increasing with $\kappa(0)=0$. A CBF is a minimal CBF \cite{konda2020characterizing} if the class-$\kappa$ function is also a minimal function\cite[Theorem 2]{konda2020characterizing}. In this paper, we consider linear extended class-$\kappa$ functions as the minimal function such that $\kappa(h(\mathbf{x}))=\kappa h(\mathbf{x})$ with $\kappa>0$. 

The safe controller $\mathbf{k}:\mathcal{X}\rightarrow\mathcal{U}$ considered in this paper is based on the CBF-Quadratic Program shown below
\begin{subequations}
\label{eq:CBF-QP}
\begin{align}
    \mathbf{k}(\mathbf{x}) = arg\min_{\mathbf{u}}& \frac{1}{2}\mathbf{u}^T\mathbf{H}(\mathbf{x})\mathbf{u} + \mathbf{c}^T(\mathbf{x})\mathbf{u}\\
    \text{s.t. }& L_{\mathbf{f}}h(\mathbf{x}) + L_{\mathbf{G}}h(\mathbf{x})\mathbf{u} \geq -\kappa h(\mathbf{x}) \label{eq:CBF-QP-Safety_const}\\
    & \mathbf{A}\mathbf{u} \leq \mathbf{c}
\end{align}
\end{subequations}
where $\mathbf{H}(\mathbf{x})\succ 0$, $\forall \mathbf{x}\in\mathcal{X}$. By \cite[Theorem 5.4]{mestres2025regularity}, the following theorem provides a regularity property for the CBF-QP above to ensure formal safety.
\begin{lemma}[\cite{mestres2025regularity}]
\label{thm:minimal_CBF+forward_inv}
    The closed loop system \eqref{eq:sys_dyn_control_aff} with controller $\mathbf{u}=\mathbf{k}(\mathbf{x})$ is forward invariant with respect to $\mathcal{C}$ if $h(\mathbf{x})$ is a minimal CBF and the controller $\mathbf{k}(\mathbf{x})$ is continuous. 
\end{lemma}

By \cite{mestres2025regularity}, the controller $\mathbf{k}(\mathbf{x})$ is point-Lipschitz and thus continuous if the following Slater's Condition holds: $\forall \mathbf{x}\in\mathcal{C}$, $\exists\mathbf{u}$ such that $\mathbf{A}\mathbf{u}<\mathbf{c}$ and the inequality constraint \eqref{eq:CBF-QP-Safety_const} strictly holds.

\subsection{Algebraic Background}
In this subsection, we first introduce Farkas' lemma, which gives a necessary and sufficient condition for the existence of solutions to linear constraints. Then, we introduce a Sum-Of-Squares (SOS) method from \cite{dai2023convex} and \cite{dai2024verification} to verify emptiness of a set. 
\begin{lemma}[Farkas' Lemma \cite{matousek2006understanding}]
\label{lemma:farkas}
    Let $\boldsymbol{\Lambda} \in \mathbb{R}^{m\times n}$ be a matrix and $\boldsymbol{\xi} \in \mathbb{R}^{m}$ be a vector. $\exists \mathbf{x}$ with  $\boldsymbol{\Lambda} \mathbf{x}\leq\boldsymbol{\xi}$ if and only if the set $\{\mathbf{z}|\mathbf{z}\geq\mathbf{0},  \boldsymbol{\Lambda}^T\mathbf{z} = \mathbf{0}, \boldsymbol{\xi}^T\mathbf{z}=-1\}$ is empty.
\end{lemma}

In this paper, we use the equivalent condition:  $\boldsymbol{\Lambda} \mathbf{x}\leq\boldsymbol{\xi}$ has a solution $\mathbf{x}$ if and only if the set $\{\mathbf{y}|\boldsymbol{\Lambda}^T\mathbf{y}^2 = \mathbf{0}, \boldsymbol{\xi}^T\mathbf{y}^2=-1\}$ is empty.

Now we introduce the theorem in \cite{dai2024verification} that verifies the emptiness of a given set. Note that a polynomial $s(\mathbf{x})$ is an SOS if $s(\mathbf{x})$ can be decomposed as $s(\mathbf{x}) = \sum_{i=1}^{N}p^2_i(\mathbf{x})$ for a set of polynomials $p_1(\mathbf{x}),\dots,p_N(\mathbf{x})$. 
\begin{lemma}[\cite{dai2024verification}]
\label{lemma:empty_verif}
    For some given polynomials $\phi_{i}(\mathbf{x})$, $i=1,\dots,N$, and polynomials $\psi_{j}(\mathbf{x})$, $j=1,\dots,M$, we define the set $\mathcal{S}$ as follows
    \begin{align*}
        \mathcal{S} = \left\{
        \begin{array}{l|ll}
        \mathbf{x} \in \mathbb{R}^{n} &
        \begin{array}{ll}
        \phi_{i}(\mathbf{x})=0,\forall i=1,\ldots,N\\
        \psi_{j}(\mathbf{x}) \geq 0,\forall j=1,\ldots,M\\
        \end{array}
        \end{array}
        \right\}
    \end{align*}
    Set $\mathcal{S}$ is empty if the following SOS program is feasible
    \begin{align*}
        \text{Find }& s_1(\mathbf{x}),\dots,s_N(\mathbf{x}) \text{ and }q_1(\mathbf{x}),\dots,q_M(\mathbf{x})\\
        \text{s.t. }&
        -1-\sum_{i=1}^{N}s_i(\mathbf{x})\phi_i(\mathbf{x}) - \sum_{j=1}^{M}q_j(\mathbf{x})\psi_j(\mathbf{x}) \text{ is SOS}\\
        & s_1(\mathbf{x}),\dots,s_N(\mathbf{x}) \text{ are SOS}
    \end{align*}
\end{lemma}
where $q_1(\mathbf{x}),\dots,q_M(\mathbf{x})$ are polynomials.

Lemma \ref{lemma:empty_verif} provides a sufficient condition for a set to be empty. With this SOS method, the verification of an empty set can be transformed into an SOS program that can be solved by SOSTOOLS\cite{prajna2002introducing} or MOSEK\cite{aps2022mosek} solvers. 

\section{Forward Invariance Conditions}
\label{sec:fi_condition}

This section presents our first result of this paper. Let $h_1(\mathbf{x}),\dots,h_{N_h}(\mathbf{x})$ be a given collection of differentiable functions. We define set $\mathcal{P}=\{1,\dots,N_h\}$, and define $\mathcal{X}_p = \{\mathbf{x}|h_p(\mathbf{x})\geq 0\}$ for all $p\in\mathcal{P}$. Our goal is to keep the system \eqref{eq:sys_dyn_control_aff} controlled forward invariant with respect to the set $\mathcal{X}_c = \cup_{p=1}^{N_h} \mathcal{X}_p$. We first introduce the considered switching CBF-QP, then give sufficient conditions for controlled forward invariance under union-CBFs.  

\subsection{Switching CBF-QPs}
\label{sec:switching-cbf-qp}
To achieve controlled forward invariance, we let the control policy switch between controllers $\mathbf{k}_1(\mathbf{x}),...,\mathbf{k}_{N_h}(\mathbf{x})$, where each $\mathbf{k}_p(\mathbf{x}), p\in\mathcal{P}$ is computed as
\begin{subequations}
\label{eq:CBF-QP-subcbf}
\begin{align}
    \mathbf{k}_p(\mathbf{x}) = arg\min_{\mathbf{u}}& \frac{1}{2}\mathbf{u}^T\mathbf{H}(\mathbf{x})\mathbf{u} + \mathbf{c}^T(\mathbf{x})\mathbf{u}\\
    \text{s.t. }& L_{\mathbf{f}}h_p(\mathbf{x}) + L_{\mathbf{G}}h_p(\mathbf{x})\mathbf{u} \geq -\kappa h_p(\mathbf{x})\label{eq:CBF-QP-safety}\\
    & \mathbf{A}\mathbf{u}\leq\mathbf{c} \label{eq:CBF-QP-contollim}
\end{align}
\end{subequations}
where $\mathbf{H}(\mathbf{x})\succ 0$ for all $\mathbf{x}\in\mathcal{X}$. Let $\sigma:[0,\infty)\rightarrow\mathcal{P}$ be the switching policy. A switching CBF-QP controller is then denoted as $\mathbf{u}(t)=\mathbf{k}_{\sigma(t)}(\mathbf{x}(t))$. 
In this paper, we assume a compact region $\mathcal{R}\subseteq\mathcal{X}$ as a region of interest such that, when $\mathbf{x}(t)\in\mathcal{R}$, we use switching CBF-QP controller \eqref{eq:CBF-QP-subcbf} to keep system \eqref{eq:sys_dyn_control_aff} safe from $\mathcal{X}_u$. We also assume $\mathcal{R}$ is large enough so that when $\mathbf{x}(t)$ enters $\mathcal{R}$, it will stay in $\mathcal{R}$ for at least $T>0$ length of time. Next, we consider the control system \eqref{eq:sys_dyn_control_aff} with a switching CBF-QP controller and propose sufficient conditions that keep \eqref{eq:sys_dyn_control_aff} controlled forward invariant in $\mathcal{X}_c$.

\subsection{Sufficient Conditions For Union-CBFs}
\label{sec:general_sufficient_condition}
\begin{proposition}
\label{prop:fi_switch_control}
    Consider control affine model \eqref{eq:sys_dyn_control_aff} with switching CBF-QP controller $\mathbf{u}(t) = \mathbf{k}_{\sigma(t)}(\mathbf{x}(t))$ suppose that:
    {
    \begin{itemize}
        \item[\textbf{P1-1}] $\sigma(t)$ is right continuous, piecewise constant, and only has a finite number of switches for any finite time horizon $[0, T]$.
        \item[\textbf{P1-2}] $\forall t\geq 0$, $\sigma(t)$ satisfies $h_{\sigma(t)}(\mathbf{x}(t))\geq 0$, and $\exists \mathbf{u} \in Int(\mathcal{U})$ with 
        \begin{equation}
        \label{eq:strict_safety}
              L_{\mathbf{f}}h_{\sigma(t)}(\mathbf{x}(t)) + L_{\mathbf{G}}h_{\sigma(t)}(\mathbf{x}(t))\mathbf{u}>-\kappa h_{\sigma(t)}(\mathbf{x}(t))
        \end{equation}
        where $\kappa$ is a positive constant.
    \end{itemize}
    }
    Then, $\forall \mathbf{x}'\in\mathcal{X}_c$, $\mathbf{x}(0) = \mathbf{x}'$ implies $\mathbf{x}(t)\in\mathcal{X}_c, \forall t\geq 0$.
\end{proposition}
\begin{proof}
    For any $T>0$, we divide the time range $[0, T]$ into $[0, T]=\cup_{k=1}^{K+1}[t_{k-1}, t_k)$ where $t_0=0$, $t_{K+1}=T$, and $\sigma(t)$ is constant in time interval $[t_{k-1}, t_k)$, $\forall k=1,\dots,K+1$. Let $\mathcal{K}=\{1,...,K\}$. So that switches happen at $t_k, \forall k\in\mathcal{K}$. For all $t \in [t_{k-1}, t_k)$ where $k=0,\dots,K+1$, $\sigma(t)=\sigma(t_{k-1})\in\mathcal{P}$. According to P1-2, constraints \eqref{eq:CBF-QP-safety} and \eqref{eq:CBF-QP-contollim} always satisfy the Slater's condition when $\sigma(t)=p$. This means for each $k=0,\dots,K$ and for all $t\in[t_k, t_{k+1})$, the controller $\mathbf{k}_{\sigma(t)}(\mathbf{x})$ is point-Lipschitz. Hence, the closed-loop system \eqref{eq:sys_dyn_control_aff} with controller $\mathbf{u}=\mathbf{k}_{\sigma(t)}(\mathbf{x})$ is a switched system \cite[Part-III]{liberzon2003switching} whose dynamics is continuous with respect to $\mathbf{x}$ in between switches. This implies that the closed-loop trajectory $\mathbf{x}(t)$ is continuously differentiable in between switches. Hence, the closed-loop dynamics would be piecewise continuous with respect to $t$ for all $t\geq 0$. Hence, we have that the closed-loop trajectory $\mathbf{x}(t)$ is absolutely continuous for all $t\geq 0$ \cite[Chapter~1.2]{liberzon2003switching}.
    
    \textbf{During the time interval $[t_0, t_1)$: } By condition P1-1, $\forall t \in[t_0, t_1)$, $\sigma(t) = \sigma(t_0)$, and the controller for the system \eqref{eq:sys_dyn_control_aff} is $\mathbf{u}=\mathbf{k}_p(\mathbf{x})$ with $p\in\mathcal{P}$ such that $h_p(\mathbf{x}(t_0))\geq 0$. By condition 1-2, the constraints for controller $\mathbf{k}_p(\mathbf{x})$ satisfy Slater's Conditions. Also, the CBF $h_p(\mathbf{x})$ is a minimal CBF since \eqref{eq:strict_safety} requires a linear extended class-$\kappa$ function. Hence, by Theorem \ref{thm:minimal_CBF+forward_inv}, $\mathbf{x}(t_0)\in\mathcal{X}_c$ implies $\mathbf{x}(t)\in\mathcal{X}_p\subseteq\mathcal{X}_c$, $\forall t\in[t_0, t_1)$, where $p$ satisfies $h_p(\mathbf{x}(t_0))\geq 0$. 

    \textbf{During the time interval $[t_k, t_{k+1})$: } We assume that $\mathbf{x}(t)\in\mathcal{X}_c$ for all $t\in[t_{k-1}, t_{k})$. Since $\mathbf{x}(t)$ is absolutely continuous, then $\lim_{t\rightarrow t_{k}^+}\mathbf{x}(t) = \lim_{t\rightarrow t_{k}^-}\mathbf{x}(t)=\mathbf{x}(t_k)\in\mathcal{X}_c$. According to P1-1, $\sigma(t)=\sigma(t_k)=p'\in\mathcal{P}$ for all $t\in[t_{k}, t_{k+1})$, where $p'$ satisfies $h_{p'}(\mathbf{x}(t_k))\geq 0$. By P1-2, the control constraints for $\mathbf{k}_{p'}(\mathbf{x})$ also satisfy Slater's Condition, and $h_{p'}(\mathbf{x})$ is a minimal CBF. Hence, $\mathbf{x}(t)\in\mathcal{X}_{p'}\subseteq\mathcal{X}_c$ for all $t\in[t_{k}, t_{k+1})$ 

    By induction, $\mathbf{x}(t)\in\mathcal{X}_c$ for all $t\in[0, T]$, and for all $T>0$.
\end{proof}

The conditions in Proposition \ref{prop:fi_switch_control} not only depends on $h_1(\mathbf{x}),...,h_{N_h}(\mathbf{x})$, but also depends on the switching signal $\sigma(t)$. Hence, in the next section, in order to verify Proposition \ref{prop:fi_switch_control}, we first specify the switching strategy, then propose some conditions for $h_1(\mathbf{x}),\dots,h_{N_h}(\mathbf{x})$ such that the switching signal $\sigma(t)$ together with $h_1(\mathbf{x}),\dots, h_{N_h}(\mathbf{x})$ satisfy P1-1 to P1-2. 
\section{Forward Invariance Conditions Under Different Switching Strategies}
\label{sec:strategies}

In this section we provide two switching strategies and provide conditions for forward invariance under these switching strategies. All sufficient conditions proposed in this section depends only on $h_1(\mathbf{x})$, $\dots$, $h_{N_h}(\mathbf{x})$. 

Before we introduce the details of the switching strategies and the corresponding forward invariance conditions, we first provide a result that will be used during the following discussion of this paper.

\begin{theorem}
\label{thm:positive_switching_dist}
     Let $f: \mathbb{R}^{n_x}\rightarrow\mathbb{R}$ be a locally Lipschitz function with Lipschitz constant $L$ on a compact set $\mathcal{W}\subseteq\mathbb{R}^{n_x}$. Let $\eta > 0$ be a constant. For all $\mathbf{x}, \mathbf{x}'\in\mathcal{W}$, we have that 
     \begin{equation*}
         \inf \left\{ ||\mathbf{x}'-\mathbf{x}|| : f(\mathbf{x})\geq \eta, f(\mathbf{x}')\leq 0 \right\} > \frac{\eta}{L}
     \end{equation*}
\end{theorem}
\begin{proof}
     For all $\mathbf{x},\mathbf{x}'\in\mathcal{W}$ such that $f(\mathbf{x})\geq \eta$ and $f(\mathbf{x}')\leq 0$, we have $\eta \leq |f(\mathbf{x}) - f(\mathbf{x}')|$. Now considering the locally lipschitz property of $f(\mathbf{x})$, we have
     \begin{align*}
         \eta \leq |f(\mathbf{x}) - f(\mathbf{x}')| \leq L||\mathbf{x}-\mathbf{x}'||
         \implies \eta \leq L||\mathbf{x}-\mathbf{x}'||
     \end{align*}
     Hence, $\frac{\eta}{L}$ is the lower bound for the distance from any $\mathbf{x}\in\mathcal{W}$ with $f(\mathbf{x})\geq\eta$ to any other $\mathbf{x}'\in\mathcal{W}$ with $f(\mathbf{x}')\leq0$.
\end{proof}

\subsection{Switching Strategy I}
\label{sec:switching_policy_i}
We give the first switching strategy as follows. Let $\eta_h > \eta_l >0$ be constants. The signal $\sigma(t)$ switches from $p\in\mathcal{P}$ to another $p'\in\mathcal{P}$ at time $t$ if the following conditions are satisfied.
\begin{itemize}
     \item[\textbf{S1-1}] $h_p(\mathbf{x}(t))\geq 0$ and $h_{p'}(\mathbf{x}(t))\geq 0$.
     \item[\textbf{S1-2}] At state $\mathbf{x}=\mathbf{x}(t)$,
     \begin{equation*}
         \sup_{\mathbf{u}\in Int(\mathcal{U})}\{L_\mathbf{f}h_p(\mathbf{x}) + L_{\mathbf{G}}h_p(\mathbf{x})\mathbf{u}+\kappa h_p(\mathbf{x})\}\leq \eta_l.
     \end{equation*}
     \item[\textbf{S1-3}] At state $\mathbf{x}=\mathbf{x}(t)$, $\exists \mathbf{u}\in Int(\mathcal{U})$ such that
     \begin{equation*}
         L_\mathbf{f}h_{p'}(\mathbf{x}) + L_{\mathbf{G}}h_{p'}(\mathbf{x})\mathbf{u}+\kappa h_{p'}(\mathbf{x}) \geq \eta_h.
     \end{equation*}
\end{itemize}
{The idea of this switching strategy is that the CBF-QP switches to another CBF constraint when the current CBF-QP's feasible control set is too small.} The following theorem gives conditions for forward invariance under this switching strategy.
 
\begin{theorem}
\label{thm:union-CBF-switching 1}
     Let $h_1(\mathbf{x})$, $\dots$, $h_{N_h}(\mathbf{x})$ be given differentiable functions with their first order derivatives being locally Lipschitz on $\mathcal{X}_c$. {Also let $\epsilon_{cbf}>0$ be a small constant and $\epsilon_{cbf}\leq \eta_l$.} If $\forall \mathbf{x}\in\mathcal{X}_c$, $\exists p\in\mathcal{P}, \mathbf{u}\in Int(\mathcal{U})$ with 
     \begin{itemize}
         \item[\textbf{C1-1}] $h_p(\mathbf{x})\geq 0$.
         \item[\textbf{C1-2}] $L_\mathbf{f}h_p(\mathbf{x}) + L_{\mathbf{G}}h_p(\mathbf{x})\mathbf{u}+\kappa h_p(\mathbf{x}) \geq \epsilon_{cbf}$.
     \end{itemize}
    Then, any switching signal $\sigma(t)$ following conditions S1-1 to S1-3 satisfies P1-1 and P1-2.
\end{theorem}

\begin{proof}
    We first prove P1-1 and then prove P1-2.

    \textit{\underline{Proof of P1-1: }} Assume $\sigma(t)$ switches to $p$ at time $t_k$, and the next switch of $\sigma(t)$ from $p$ to another index in $\mathcal{P}$ happens at $t_{k+1}$. In order to prove P1-1, we first prove that there exists a uniform lower bound for $||\mathbf{x}(t_{k+1}) - \mathbf{x}(t_k)||$, then we show that there exists a uniform lower bound on $t_{k+1}-t_{k}$ that holds for all $k=0,1,\ldots$. According to the problem setup, $\mathbf{f}(\mathbf{x})$ and $\mathbf{G}(\mathbf{x})$ of system \eqref{eq:sys_dyn_control_aff} are locally Lipschitz for all $\mathbf{x}\in\mathcal{X}$. 
    For all $p\in\mathcal{P}$, we let $f_p(\mathbf{x})=L_\mathbf{f}h_p(\mathbf{x}) + L_{\mathbf{G}}h_p(\mathbf{x})\mathbf{u}_p+\kappa h_p(\mathbf{x})$, where $\mathbf{u}_p\in\mathcal{U}$ is an admissible control such that $f_p(\mathbf{x}(t), \mathbf{u}_p)\geq \eta_h$. Since $\mathcal{U}$ is a compact set, we have $||\mathbf{u}_p||\leq U$ where $U=\max_{\mathbf{u}\in\mathcal{U}}||\mathbf{u}||$. We denote $L_{p}^{(G)}$ and $L_{p}^{(f)}$ be the Lipschitz constants of terms $(L_\mathbf{f}h_p(\mathbf{x})+\kappa h_p(\mathbf{x}))$ and $L_{\mathbf{G}}h_p(\mathbf{x})$ in the compact region $\mathcal{R}$, respectively. Hence, for all $\mathbf{x}\neq\mathbf{x}'$, we have:
    \begin{align*}
         \|f_p(\mathbf{x}) - f_p(\mathbf{x}')\|
         \leq (L_p^{(G)} + L_p^{(f)}U)\|\mathbf{x}-\mathbf{x}'\|
    \end{align*}
    We let $L_p = L_p^{(G)} + L_p^{(f)}U$ and $L=\max_{p\in\mathcal{P}}L_p$. 
    By Theorem \ref{thm:positive_switching_dist}, for any $\mathbf{x}'\in\mathcal{X}_p$ such that $f_{p}(\mathbf{x}')\leq \eta_l$, we have $||\mathbf{x}' - \mathbf{x}(t_k)||\geq \frac{\eta_h-\eta_l}{L}$. Define the following two sets: 
    \begin{align*}
        \mathcal{Y} &:= \{\mathbf{x}|\sup_{\mathbf{u}\in Int(\mathcal{U})}\{L_\mathbf{f}h_p(\mathbf{x}) + L_{\mathbf{G}}h_p(\mathbf{x})\mathbf{u}+\kappa h_p(\mathbf{x})\}\leq \eta_l\}\\
        \mathcal{Y}_0 &:= \{\mathbf{x}| L_\mathbf{f}h_p(\mathbf{x}) + L_{\mathbf{G}}h_p(\mathbf{x})\mathbf{u}_p+\kappa h_p(\mathbf{x}) \leq \eta_l\}
    \end{align*}
    where $\mathbf{u}_0\in Int(\mathcal{U})$. It is obvious that $\mathcal{Y}\subseteq\mathcal{Y}_0$. Since we have $\mathbf{x}(t_k)\notin\mathcal{Y}$, $\mathbf{x}(t_k)\notin\mathcal{Y}_0$, and $\mathbf{x}(t_{k+1})\in\mathcal{Y}$, we have that
    \begin{equation*}
        ||\mathbf{x}(t_{k+1})- \mathbf{x}(t_k)||\geq||\mathbf{x}'- \mathbf{x}(t_k)||\geq {\frac{\eta_h-\eta_l}{L}}
    \end{equation*}
    Hence, we have a uniform lower bound for the term $||\mathbf{x}(t_{k+1})- \mathbf{x}(t_k)||\geq \frac{\eta_h-\eta_l}{L}$. {Since we also assume that $\mathcal{R}$ and $\mathcal{U}$ are bounded sets, there exists a constant $D>0$ such that $||\mathbf{f}(\mathbf{x}(t)) + \mathbf{G}(\mathbf{x}(t))\mathbf{u}(t)|| \leq D$, $\forall t\geq 0$. Hence, we now have a uniform lower bound for $||t_{k+1}-t_k||\geq \frac{\eta_h-\eta_l}{DL}$ for all $k=0,1,...$.} 
    Condition P1-1 is now proved.

    \textit{\underline{Proof of P1-2: }} If conditions C1-1 and C1-2 are satisfied, then all possible switching signals under Switching Strategy 1 satisfy P1-2. Hence, P1-2 is proved.
    
\end{proof}

\subsection{Switching Strategy II}
\label{sec:switching_policy_ii}
Let $\eta_h, \eta_l, \tau$ be positive constants and $\eta_h > \eta_l$. We now consider the following switching strategy:
{
\begin{itemize}
    \item $\forall t\in[0, \tau)$, $\sigma(t) = p_0$ such that $h_{p0}(x(0))\geq 0$.
    \item $\sigma(t)$ switches from a $p\in\mathcal{P}$ to another $p'\in\mathcal{P}$ at time $t$ if $h_p(\mathbf{x}(t))\leq \eta_l$ and $h_{p'}(\mathbf{x}(t))\geq \eta_h$
\end{itemize}
}
Under the Switching Strategy II, we have the union CBF condition provided in the following theorem.
\begin{theorem}
\label{thm:union-CBF-switching 2}
    Let $h_1(\mathbf{x}),\dots,h_{N_h}(\mathbf{x})$ be a collection of given polynomials. {Let $\epsilon_{cbf}>0$ be a constant.} If $\forall p\in\mathcal{P}$ and $\forall \mathbf{x}\in\mathcal{X}_p$, $\exists \mathbf{u}\in Int(\mathcal{U})$ such that
    \begin{itemize}
        \item [\textbf{C2-1}] $L_\mathbf{f}h_p(\mathbf{x}) + L_{\mathbf{G}}h_p(\mathbf{x})\mathbf{u}+\kappa h_p(\mathbf{x}) \geq \epsilon_{cbf}$.
    \end{itemize}
    Then, any switching signal $\sigma(t)$ following the Switching Strategy 2 satisfies conditions P1-1 and P1-2.
\end{theorem}
\begin{proof}
    We prove Theorem \ref{thm:union-CBF-switching 2} by considering the following two cases.
    \begin{itemize}
        \item [(a)]$\forall \mathbf{x}\in\mathcal{X}_c$, $\nexists p, p'\in\mathcal{P}$ such that $p\neq p'$, $h_p(\mathbf{x})\leq \eta_l$, and $h_{p'}(\mathbf{x})\geq \eta_h$.
        \item [(b)]$\exists \mathbf{x}\in\mathcal{X}_c$ such that $h_p(\mathbf{x})\leq \eta_l$ and $h_{p'}(\mathbf{x})\geq \eta_h$ for some $p, p'\in\mathcal{P}$ and $p \neq p'$. 
    \end{itemize}
    
    \textit{\underline{Proof of P1-1: }}
    According to condition C2-1, in case (a), once the switching signal $\sigma(t)$ is initialized as $p_0$ such that $h_{p_0}(\mathbf{x}(0))\geq0$, the closed-loop system \eqref{eq:sys_dyn_control_aff} with CBF-QP controller $\mathbf{u}_{p_0}(\mathbf{x})$ is forward invariant to set $\{\mathbf{x}|h_{p_0}(\mathbf{x})\geq 0\}\subseteq\mathcal{X}_c$. Hence, the switching signal keeps its value $\sigma(t)=p_0$ for all $t\geq 0$. 
    
    In case (b), the switching signal $\sigma(t)$ may switch at some $t > \tau$. In this case, the positive constant $\tau$ ensures that $\sigma(t)$ does not have a switch in an arbitrary small time range after the initialization $\sigma(0)=p_0$. After the time $\tau$, we let $t_{k} >\tau$ be the time slot when $\sigma(t)$ switches to $p\in\mathcal{P}$ such that $h_p(\mathbf{x}(t_k))\geq\eta_h$. We also denote $t_{k+1}$ as the time when $h_p(\mathbf{x}(t_{k+1}))\leq \eta_l$ and $\sigma(t)$ switches away from $p$. Since, $\forall p\in\mathcal{P}$, $h_p(\mathbf{x})$ is a polynomial and is locally Lipschitz on $\mathcal{X}$. Let the Lipschitz constant be $L_p$, and let $L_h = \max_{p\in\mathcal{P}}L_p$. By Theorem \ref{thm:positive_switching_dist}, we always have that $||\mathbf{x}(t_{k+1}) - \mathbf{x}(t_k)|| \geq \frac{\eta_h - \eta_l}{L_p}$. Since we assumed that $\mathcal{R}$ and $\mathcal{U}$ are bounded sets, then there exists constant $D>0$ such that $||\mathbf{f}(\mathbf{x}(t)) + \mathbf{G}(\mathbf{x}(t))\mathbf{u}(t)|| \leq D$, $\forall t\geq 0$. Hence, we now have a uniform lower bound for $||t_{k+1}-t_k||\geq \frac{\eta_h - \eta_l}{DL_h}$ for all $k=0,1,...$. Hence, we prove the P1-1 for Theorem \ref{thm:union-CBF-switching 2} in case (b). 
    
    \textit{\underline{Proof of P1-2: }}
    Since $\sigma(t)\in\mathcal{P}$ for all $t\geq 0$, then, condition C2-1 of Theorem 3 implies P1-2 in both (a) and (b) cases.
\end{proof}  

\textbf{Remark: } Since condition C2-1 also implies C1-2, the union-CBFs condition in Theorem \ref{thm:union-CBF-switching 2} is also a sufficient condition for Theorem \ref{thm:union-CBF-switching 1}. {In the experiments, we will show an example that is verified by Theorem \ref{thm:union-CBF-switching 1} but can not be verified by Theorem \ref{thm:union-CBF-switching 2}.}
\section{Verification Framework}
\label{sec:verification}
This section introduces SOS frameworks to verify sufficient conditions in Theorem \ref{thm:union-CBF-switching 1} and Theorem \ref{thm:union-CBF-switching 2}. Given a polynomial system dynamics \eqref{eq:sys_dyn_control_aff} and a collection of CBFs $h_1(\mathbf{x}),\dots,h_{N_h}(\mathbf{x})$, verifying the union of these CBFs is to verify (a) Theorem \ref{thm:union-CBF-switching 1} or Theorem \ref{thm:union-CBF-switching 2} is true, and (b) validity of these CBFs, meaning that $\mathcal{X}_c\subseteq\mathcal{X}_s$. Hence, in what follows, we first introduce the SOS verification frameworks for Theorem \ref{thm:union-CBF-switching 1}. Then we introduce the verification framework for Theorem \ref{thm:union-CBF-switching 2}. Finally, we verify that $\mathcal{X}_c\subseteq\mathcal{X}_s$.

We use $\mathcal{S}^{\mathcal{P}}$ to denote the collection of all possible subsets of $\mathcal{P}$ except $\emptyset$. During the verification of Theorem \ref{thm:union-CBF-switching 1} and Theorem \ref{thm:union-CBF-switching 2}, we let $\Bar{\mathcal{U}}=\{\mathbf{u}|\mathbf{A}\mathbf{u}\leq\mathbf{c}-\epsilon_{u}\}$ with a given constant $\epsilon_{u}>0$ so that $\Bar{\mathcal{U}}\subset Int(\mathcal{U})$. For simplicity, we also define matrix $\boldsymbol{\Lambda}_p(\mathbf{x})$ and vector $\boldsymbol{\xi}_p(\mathbf{x})$ such that
\begin{equation}
\label{eq:xi_lambda}
    \boldsymbol{\Lambda}_p(\mathbf{x}) = 
    \begin{bmatrix}
        -L_{\mathbf{G}}h_p(\mathbf{x})\\
        \mathbf{A}
    \end{bmatrix};
    \boldsymbol{\xi}_p(\mathbf{x}) = 
    \begin{bmatrix}
        L_{\mathbf{f}}h_p(\mathbf{x}) + \kappa h_p(\mathbf{x}) - \epsilon_{cbf}\\
        \mathbf{c} - \epsilon_{u}
    \end{bmatrix}
\end{equation}
where $\eta, \epsilon$ are specified positive constants. Hence, for any $\mathbf{u}$ with $\boldsymbol{\Lambda}_p(\mathbf{x})\mathbf{u}\leq\boldsymbol{\xi}_p(\mathbf{x})$, $\mathbf{u}$ strictly satisfies constraints \eqref{eq:CBF-QP-safety} and \eqref{eq:CBF-QP-contollim}.

\subsection{Framework For Switching Strategy I}
\label{sec:verification_I}
According the defined symbols above, we verify Theorem \ref{thm:union-CBF-switching 1} by solving the following problem.

\textbf{Verification Problem 1: }Given polynomials $h_1(\mathbf{x})$,$\dots$, $h_{N_h}(\mathbf{x})$, system model \eqref{eq:sys_dyn_control_aff}, positive constants $\kappa$, $\eta$, $\epsilon$ and constant matrix $\mathbf{A}$ and vector $\mathbf{c}$, verify that, $\forall\mathbf{x}\in\mathcal{X}_c$, $\exists p\in\mathcal{P}$ and $\mathbf{u}\in\mathbb{R}^{n_u}$ such that $h_p(\mathbf{x})\geq 0$ and $\Lambda_p(\mathbf{x})\mathbf{u}\leq\xi_p(\mathbf{x})$.

We solve Problem 1 by three steps. First, we partition the set $\mathcal{X}_c$. Next, we collect all the non-empty subsets of $\mathcal{X}_c$. Finally, for each of these non-empty subsets, we verify that, for all states $\mathbf{x}$ in the subset, there exists $p\in\mathcal{P}$ and $\mathbf{u}\in\mathbb{R}^{n_u}$ such that $h_p(\mathbf{x})\geq 0$ and $\Lambda_p(\mathbf{x})\mathbf{u}\leq\xi_p(\mathbf{x})$.

\textbf{\underline{Step 1:}} We partition the set $\mathcal{X}_c$ into subsets such that $\mathcal{X}_c = \cup_{\mathcal{N}\in \mathcal{S}^{\mathcal{P}}}\mathcal{X}_{\mathcal{N}}$, where $\mathcal{X}_{\mathcal{N}}$ for all $\mathcal{N}\in \mathcal{S}^{\mathcal{P}}$ is defined as
\begin{equation*}
    \mathcal{X}_{\mathcal{N}} = \{\mathbf{x}|h_p(\mathbf{x})\geq0, \forall p\in\mathcal{N}; h_{p'}(\mathbf{x})<0, \forall p'\in\mathcal{P}\setminus\mathcal{N}\}.
\end{equation*}

To solve Problem 1, we consider the following two statements.
\begin{itemize}
    \item [T1: ]$\forall \mathcal{N}\in\mathcal{S}^{\mathcal{P}}$ and $\forall \mathbf{x}\in\mathcal{X}_{\mathcal{N}}$, $\exists p\in\mathcal{N}$ and $\mathbf{u}\in\mathbb{R}^{n_u}$ such that $\Lambda_p(\mathbf{x})\mathbf{u}\leq\xi_p(\mathbf{x})$.
    \item [T2: ]$\forall \mathbf{x}\in\mathcal{X}_{\mathcal{N}}$, $\exists p\in\mathcal{N}$ and $\mathbf{u}\in\mathbb{R}^{n_u}$ such that $\Lambda_p(\mathbf{x})\mathbf{u}\leq\xi_p(\mathbf{x})$.
\end{itemize}
According to the partition for $\mathcal{X}_c$ and definition of $\mathcal{X}_{\mathcal{N}}$, verifying T1 is sufficient to verify Theorem \ref{thm:union-CBF-switching 1}, and is equivalent to iteratively verifying T2 over all $\mathcal{N}\in\mathcal{S}^{\mathcal{P}}$. 

\textbf{\underline{Step 2:}} In the most general case where $\mathcal{X}_{\mathcal{N}}\neq\emptyset$ for all $\mathcal{N}\in\mathcal{S}^{\mathcal{P}}$, the verification of T2 is repeated for $2^N-1$ times. However, in cases where $\mathcal{X}_1,...,\mathcal{X}_{N_h}$ are located sparsely, there are some $\mathcal{N}\in\mathcal{S}^{\mathcal{P}}$ such that $\mathcal{X}_{\mathcal{N}}=\emptyset$. Verifying T1 only requires to repeat the verification of T2 for all $\mathcal{X}_{\mathcal{N}}\neq\emptyset$. In this step, we check whether $\mathcal{X}_{\mathcal{N}} = \emptyset$ for all $\mathcal{N}\in\mathcal{S}^{\mathcal{P}}$ so that, at the end of this step, we should have the collection $\mathcal{S}_{ne} = \{\mathcal{N}\in\mathcal{S}^{\mathcal{P}}|\mathcal{X}_{\mathcal{N}}\neq\emptyset\}$ where the subscript $ne$ stands for ``non-empty". 

The following lemma provides an SOS program to check whether $\mathcal{X}_{\mathcal{N}}=\emptyset$.
\begin{lemma}
\label{lemma:non_empty_subset_union}
    $\forall \mathcal{N}\in\mathcal{S}^{\mathcal{P}}$, $\mathcal{X}_{\mathcal{N}} = \emptyset$ if the following SOS program is feasible.
    \begin{subequations}
    \label{eq:non_empty_subset_union}
    \begin{align}
        \text{Find } & s_p(\mathbf{x}, \mathbf{c}), \forall p\in\mathcal{N}, \text{ and }s_{p'}(\mathbf{x}, \mathbf{c}), \forall p'\in\mathcal{P}\setminus\mathcal{N}\\
        \text{s.t. } & -1 - \sum_{p\in\mathcal{N}}s_p(\mathbf{x},\mathbf{c})h_p(\mathbf{x}) \nonumber\\
        &-\sum_{p'\in\mathcal{P}\setminus\mathcal{N}}q_{p'}(\mathbf{x}, \mathbf{c})[c_{p'}^2h_{p'}(\mathbf{x})+1] \text{ is SOS}\\
        &s_p(\mathbf{x}, \mathbf{c}), \forall p\in\mathcal{N} \text{ are SOS}
    \end{align}
    \end{subequations}
    where $\mathbf{c}$ is a vector of scalar variables $c_{p'}$ with $ p'\in\mathcal{P}\setminus\mathcal{N}$, and each $q_{p'}(\mathbf{x},\mathbf{c})$is a polynomial of $\mathbf{x}$ and $\mathbf{c}$. 
\end{lemma}
\begin{proof}
    For a function $h:\mathbb{R}^{n_x}\rightarrow\mathbb{R}$, $\exists \mathbf{x}$ with $h(\mathbf{x})<0$ if and only if $\exists (\mathbf{x}, c)\in\mathbb{R}^{n_x+1}$ such that $c^2h(\mathbf{x})=-1$. Hence, $\mathcal{X}_{\mathcal{N}}$ is empty if and only if the following set is empty.
    \begin{equation*}
        \{(\mathbf{x},\mathbf{c})|h_p(\mathbf{x})\geq 0, \forall p\in\mathcal{N}; c^2_{p'}h_{p'}(\mathbf{x}) + 1=0, \forall p'\in\mathcal{P}\setminus\mathcal{N}\}
    \end{equation*}
    where $\mathbf{c}$ is a vector of scalar variables $c_{p'}$ with $ p'\in\mathcal{P}\setminus\mathcal{N}$. Then, by Theorem \ref{lemma:empty_verif}, we have the SOS program \eqref{eq:non_empty_subset_union}.
\end{proof}

\textbf{\underline{Step 3: }} In this step, we verify T2 for each $\mathcal{X}_{\mathcal{N}}$ with $ \mathcal{N}\in\mathcal{S}_{ne}$. We can see from Lemma \ref{lemma:non_empty_subset_union} that the terms $h_{p'}(\mathbf{x})<0$ in set $\mathcal{X}_{\mathcal{N}}$ could introduce extra variables to SOS programs. This complicates polynomial structures and makes SOS program harder to solve. Hence, instead of consider $\mathcal{X}_c$, in this step, we consider $\mathcal{X}'_{\mathcal{N}}$ defined as
\begin{equation*}
    \mathcal{X}'_{\mathcal{N}} = \{\mathbf{x}|h_p(\mathbf{x})\geq0, \forall p\in\mathcal{N}; h_{p'}(\mathbf{x})\leq 0, \forall p'\in\mathcal{P}\setminus\mathcal{N}\}.
\end{equation*}
Instead of verifying T2, we verify the following statement.
\begin{itemize}
    \item [T3: ]$\forall \mathbf{x}\in\mathcal{X}'_{\mathcal{N}}$, $\exists p\in\mathcal{N}$ and $\mathbf{u}\in\mathbb{R}^{n_u}$ such that $\Lambda_p(\mathbf{x})\mathbf{u}\leq\xi_p(\mathbf{x})$.
\end{itemize}
Since $\mathcal{X}_{\mathcal{N}}\subset\mathcal{X}'_{\mathcal{N}}$, then verifying T3 also verifies T2. 
 
The following theorem provides the SOS program to verify T3 for all $\mathcal{N}\in\mathcal{S}_{ne}$.
\begin{theorem}
\label{thm:verification_T3}
    For all $\mathcal{N}\in\mathcal{S}_{ne}$, $\mathcal{X}'_{\mathcal{N}}$ satisfies T3 if the following SOS program is feasible.
    \begin{subequations}
    \label{eq:T3_verification_prog}
    \begin{align}
        \text{Find } &s_p(\mathbf{x},\mathbf{y}), \mathbf{q}_p(\mathbf{x}, \mathbf{y}), r_p(\mathbf{x},\mathbf{y}), \forall p\in\mathcal{N} \text{ and } \nonumber \\
        &s_{p'}(\mathbf{x},\mathbf{y}), \forall p'\in\mathcal{P}\setminus\mathcal{N}\\
        \text{s.t. } &-1 - \sum_{p\in\mathcal{N}}s_p(\mathbf{x}, \mathbf{y})h_p(\mathbf{x}) + \sum_{p'\in\mathcal{P}\setminus\mathcal{N}} s_{p'}(\mathbf{x},\mathbf{y})h_{p'}(\mathbf{x}) \nonumber \\
        &- \sum_{p\in\mathcal{N}}\left(\mathbf{q}^T_p(\mathbf{x},\mathbf{y})\boldsymbol{\Lambda}^T_p(\mathbf{x})\mathbf{y}_p^2 + r_p(\mathbf{x},\mathbf{y})(\boldsymbol{\xi}_p(\mathbf{x})\mathbf{y}^2_p + 1)\right) \nonumber\\
        &\text{ is SOS} \label{eq:T3_verif_prog_SOS}\\
        &s_p(\mathbf{x},\mathbf{y}), \forall p\in\mathcal{N} \text{ are SOS}\\
        &s_{p'}(\mathbf{x},\mathbf{y}), \forall p'\in\mathcal{P}\setminus\mathcal{N} \text{ are SOS}
    \end{align}
    \end{subequations}
    where $\mathbf{y}$ is a vector of variables concatenated by all $\mathbf{y}_p$ for all $p\in\mathcal{N}$, each $\mathbf{q}_p(\mathbf{x},\mathbf{y})$ is a vector of polynomials, and each $r_p(\mathbf{x},\mathbf{y})$ is a polynomial. 
\end{theorem}

\begin{proof}
    The the following T4 statement is the negation of T3: 
    \begin{itemize}
        \item[T4: ]$\exists \mathbf{x}\in\mathcal{X}'_{\mathcal{N}}$ such that, $\forall p\in\mathcal{N}$, $\nexists \mathbf{u}_p$ with $\boldsymbol{\Lambda}_p(\mathbf{x})\mathbf{u}_p\leq\boldsymbol{\xi}_p(\mathbf{x})$
    \end{itemize}
    By Lemma \ref{lemma:farkas}, T4 is equivalent to the \textit{non-emptiness} of the following set.
    \begin{align*}
        &\mathcal{V} = \\
        &\left\{
        \begin{array}{l|ll}
        (\mathbf{x},\mathbf{y}) &
        \begin{array}{ll}
        h_p(\mathbf{x})\geq 0, \forall p\in\mathcal{N}\\
        -h_{p'}(\mathbf{x})\geq 0, \forall p'\in\mathcal{P}\setminus\mathcal{N}\\
        \boldsymbol{\Lambda}^T_p(\mathbf{x})\mathbf{y}_p^2 = \mathbf{0},
        \boldsymbol{\xi}^T_p(\mathbf{x})\mathbf{y}_p^2 + 1 = 0, \forall p\in\mathcal{N}
        \end{array}
        \end{array}
        \right\}
    \end{align*}
    where the $\mathbf{y}$ is a vector of variables concatenated by all $\mathbf{y}_p$ for all $p\in\mathcal{N}$. The first two lines of $\mathcal{V}$ correspond to $\mathcal{X}'_{\mathcal{N}}$, and the last line of $\mathcal{V}$ corresponds to the result of Farkas' Lemma for each $p\in\mathcal{N}$. Since T4 is the negation of T3, and T4 is equivalent to the non-emptiness of $\mathcal{V}$, then T3 is equivalent to the emptiness of $\mathcal{V}$. Now, applying Theorem \ref{lemma:empty_verif}, we get the SOS program \eqref{eq:T3_verification_prog} that  verifies T3. 
\end{proof}

To summarize, in this subsection, we verify sufficient conditions for Theorem \ref{thm:union-CBF-switching 1}. We first compute all possible subsets of $\mathcal{X}_c$. Second, we use Program \eqref{eq:non_empty_subset_union} to discard all the empty subsets of $\mathcal{X}_c$ and returns all the non-empty subsets. Finally, we verify a sufficient condition for T3 using Program \eqref{eq:T3_verification_prog} for all the non-empty subsets. 

\subsection{Framework For Switching Strategy II}
In this subsection, we propose the SOS programs that verifies Theorem \ref{thm:union-CBF-switching 2}. 

\textbf{Verification Problem 2: }Given a collection of polynomials $h_1(\mathbf{x})$, $\dots$, $h_{N_h}(\mathbf{x})$, system model \eqref{eq:sys_dyn_control_aff}, positive constants $\kappa, \eta > 0$, and constant matrix and vector $\mathbf{A}, \mathbf{c}$, verify that for all $p\in\mathcal{P}$ and $\mathbf{x}\in\mathcal{X}_p$, there exists $\mathbf{u}\in\mathbb{R}^{n_u}$ with $\boldsymbol{\Lambda}_p(\mathbf{x})\mathbf{u}\leq \boldsymbol{\xi}_p(\mathbf{x})$. 

\begin{theorem}
    Problem 2 can be verified by checking the feasibility of the following SOS program for all $p\in\mathcal{P}$.
    \begin{subequations}
    \label{eq:verify_theorem4}
    \begin{align}
        \text{ Find }& s(\mathbf{x},\mathbf{y}), \mathbf{q}(\mathbf{x},\mathbf{y}), r(\mathbf{x}, \mathbf{y})\\ 
        \text{s.t. }& -1 - s(\mathbf{x},\mathbf{y})h_p(\mathbf{x}) - \mathbf{q}(\mathbf{x},\mathbf{y})\boldsymbol{\Lambda}^T_p(\mathbf{x})\mathbf{y}^2 \nonumber \\
        & - r(\mathbf{x},\mathbf{y})(\boldsymbol{\xi}^T_p(\mathbf{x})\mathbf{y}^2+1) \text{ is SOS} \label{eq:thm4_verif_sos}\\
        & s(\mathbf{x},\mathbf{y}) \text{is SOS}
    \end{align}
    where $\mathbf{y}$ is a vector of variables, $\mathbf{q}(\mathbf{x},\mathbf{y})$ is a vector of polynomials, and $r(\mathbf{x},\mathbf{y})$ is a polynomial. 
    \end{subequations}
\end{theorem}

\begin{proof}
    By Lemma \ref{lemma:farkas}, for all $\mathbf{x}\in\mathcal{X}_p$, $\exists \mathbf{u}\in\mathbb{R}^{n_u}$ such that $\boldsymbol{\Lambda}_p(\mathbf{x})\mathbf{u}\leq \boldsymbol{\xi}_p(\mathbf{x})$ if and only if the following set is empty.
    \begin{equation*}
        \mathcal{V}_s = \{(\mathbf{x},\mathbf{y})| h_p(\mathbf{x})\geq 0, \boldsymbol{\Lambda}^T_p(\mathbf{x})\mathbf{y}^2=\mathbf{0}, \boldsymbol{\xi}^T_p(\mathbf{x})\mathbf{y}^2 + 1 = 0\}
    \end{equation*}
    Applying Theorem \ref{lemma:empty_verif} to $\mathcal{V}_s$, we get the SOS Program \eqref{eq:verify_theorem4}. Problem 2 is verified if Program \eqref{eq:verify_theorem4} is feasible for all $p\in\mathcal{P}$.
\end{proof}

\textbf{Remark: }Comparing Program \eqref{eq:verify_theorem4} to Program \eqref{eq:T3_verification_prog}, we can see the verification of Problem 2 is in fact a special case for verification of Problem 1. During the verification of Problem 1, if we replace $\mathcal{N}$ as the index set $\{p\}$, and replace $\mathcal{X}_{\mathcal{N}}$ as $\mathcal{X}_p = \{\mathbf{x}|h_p(\mathbf{x})\geq 0\}$, then Program \eqref{eq:T3_verification_prog} and Program \eqref{eq:verify_theorem4} would be the same. This observation coincides with the remark in Section \ref{sec:switching_policy_ii} that Theorem \ref{thm:union-CBF-switching 2} is a sufficient condition of Theorem \ref{thm:union-CBF-switching 1}.

\subsection{Verifying Validity of CBFs}
In this subsection, we verify that $\mathcal{X}_c \subseteq \mathcal{X}_s$. Assume the unsafe region $\mathcal{X}_u = \mathcal{X}\setminus\mathcal{X}_s$ is identified by a collection of polynomials such that $\mathcal{X}_u=\{\mathbf{x}|\mathbf{l}(\mathbf{x})\leq 0\}$ where $\mathbf{l}(\mathbf{x})$ is a vector of polynomials. In order to verify that $\mathcal{X}_c \subseteq \mathcal{X}_s$, we need to verify that $\mathcal{X}_c\cap\mathcal{X}_u=\emptyset$. Since $\mathcal{X}_c = \cup_{p\in\mathcal{P}}\mathcal{X}_p$, then verifying $\mathcal{X}_c\cap\mathcal{X}_u=\emptyset$ is equivalent to verify $\mathcal{X}_p\cap\mathcal{X}_u = \emptyset$ for all $p\in\mathcal{P}$. Hence, to verify $\mathcal{X}_c \subseteq \mathcal{X}_s$, we verify that the set $\{\mathbf{x}|h_p(\mathbf{x})\geq 0, -\mathbf{l}(\mathbf{x})\geq 0\}$ is empty for all $p\in\mathcal{P}$. By Theorem \ref{lemma:empty_verif}, this is to solve the following SOS feasibility program.
\begin{subequations}
\label{eq:verify_validity_prog}
\begin{align}
    \text{ Find }&s(\mathbf{x}), \mathbf{s}_{l}(\mathbf{x}) \\
    \text{ s.t. }& -1 - s(\mathbf{x})h_p(\mathbf{x}) - \mathbf{s}^T_l(\mathbf{x})\mathbf{l}(\mathbf{x}) \text{ is SOS}\\
    &s(\mathbf{x}), \mathbf{s}_{l}(\mathbf{x}) \text{ are SOS}
\end{align}
\end{subequations}
where $\mathbf{s}_l(\mathbf{x})$ is a vector of SOS with the same size as vector $\mathbf{l}(\mathbf{x})$. {If $\mathcal{X}_u$ can not be approximated by polynomials, we can also use sampling based methods to verify the validity. Assume we have a set of unsafe states $\{\mathbf{x}_1,\dots,\mathbf{x}_N\}$ sampled from $\mathcal{X}_u$, then $h_1(\mathbf{x}),...,h_{N_h}(\mathbf{x})$ are valid if $h_p(\mathbf{x}_i) - c\geq 0$ for all $p\in\mathcal{P}$ and $i=1,2,...,N$. Here $c > 0$ is a specified constant that avoids the possible overlap between $\mathcal{X}_p$ and the actual $\mathcal{X}_u$ due to the sampling error.}

\section{Experiments}
\label{sec:experiments}
This section presents experiments that support our motivation and compares the efficiency of our proposed verification frameworks. The SOS verification programs are solved using MOSEK \cite{aps2022mosek} for all examples. For simplicity, we name the verification methods for Theorem \ref{thm:union-CBF-switching 1} and Theorem \ref{thm:union-CBF-switching 2} as Verif-I and Verif-II, respectively.

\subsection{Motivating Example}
\label{sec:motivating_example}

\begin{figure}[h]
    \centering
    \begin{subfigure}{0.23\textwidth}
        \includegraphics[width=\textwidth]{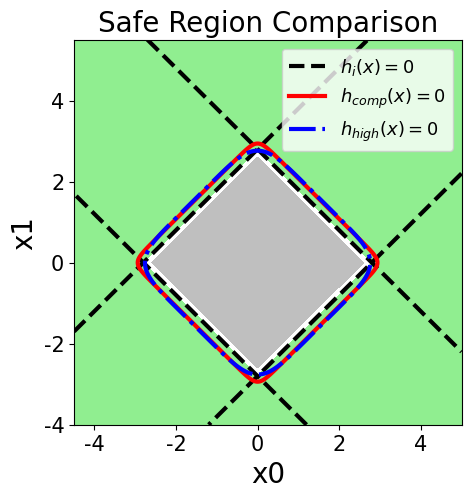}
        \caption{}
        \label{fig:coverage_comparison}
    \end{subfigure}
    \begin{subfigure}{0.23\textwidth}
        \includegraphics[width=\textwidth]{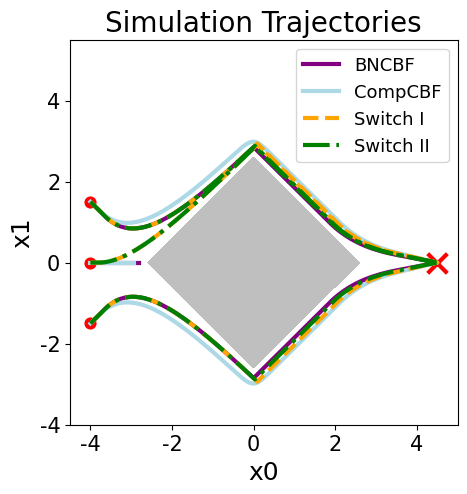}
        \caption{}
        \label{fig:trajectory_comparison}
    \end{subfigure}
    \caption{Comparison between the proposed switching union-CBFs and existing methods. Fig.\ref{fig:coverage_comparison} shows the safe region coverage of the union of 4 linear CBFs, composed CBF $h_{comp}(\mathbf{x})$ and a high-degree polynomial CBF $h_{high}(\mathbf{x})$. Fig.\ref{fig:trajectory_comparison} shows the performance comparison between safety filters using BNCBF, composed CBF, and the proposd switching union-CBFs with switching strategy I and II. In Fig \ref{fig:coverage_comparison}, the safe region $\mathcal{X}_c$ covered by the union CBFs is marked as green while $\mathcal{X}_u$ is marked as gray. Both the safe region $\mathcal{X}_{high}$ and $\mathcal{X}_{comp}$ are subsets of $\mathcal{X}_c$. In Fig. \ref{fig:trajectory_comparison}, the red dot denote initial states, and the red cross marks the goal state. At some states, composed CBF filter may get stuck due to expected equilibria; while BNCBF may get stuck due to over constraining. However, starting from the same initial state, the switching union-CBFs can avoid the sticking point and reach the goal. }
\label{fig:union_vs_high_degree}
\end{figure}

In this example, we show that a high-degree polynomial CBF and a composed CBF only covers the subset of the safe region of union-CBFs. We denote $h_{comp}(\mathbf{x})$ as the composed CBF of a collection of linear union-CBFs, and denote $h_{high}(\mathbf{x})$ as the high degree CBF. Let $\mathcal{X}_{comp} = \{\mathbf{x}|h_{comp}(\mathbf{x})\geq 0\}$, and $\mathcal{X}_{high} = \{\mathbf{x}|h_{high}(\mathbf{x})\geq 0\}$


In Fig. \ref{fig:union_vs_high_degree}, we consider a 2D single integrator $\dot{\mathbf{x}}=\mathbf{u}$ with control limits $-1 \leq \mathbf{u}_{(1)}\leq 1$ and $-1 \leq \mathbf{u}_{(2)}\leq 1$. The unsafe region has a square shape such that 
$\mathcal{X}_u = 
\{\mathbf{x}|
\pm \mathbf{x}_{(0)} \pm \mathbf{x}_{(1)} + 2.9 >=0 \}$. 
We choose 4 linear CBFs as 
$h_1(\mathbf{x})=-\mathbf{x}_{0} + \mathbf{x}_{(1)}-3$, 
$h_2(\mathbf{x})=\mathbf{x}_{0} + \mathbf{x}_{(1)}-3$, $h_3(\mathbf{x})=\mathbf{x}_{0} - \mathbf{x}_{(1)}-3$, and $h_4(\mathbf{x})=-\mathbf{x}_{0} - \mathbf{x}_{(1)}-3$. The composed CBF $h_{comp}(\mathbf{x})$ composes the 4 linear CBFs using method \cite{molnar2023composing}. The degree-4 CBF $h_{high}(\mathbf{x}) = (\mathbf{x}_{(0)} + \mathbf{x}_{(1)})^4 + (\mathbf{x}_{(0)} - \mathbf{x}_{(1)})^4 + 2.9^4$. Fig.\ref{fig:coverage_comparison} shows that the union-CBFs covers more safe states in the state space. For verifying the union-CBFs, Verif-I takes 48.95 seconds, while Verif-II only takes 0.16 seconds. On the other hand, the verification of $h_{high}(\mathbf{x})$ takes 197.21 seconds. Hence, compared to union of linear CBFs, a high-degree polynomial covers the subset of the union safe region, and also has longer verification time.

Using the same example setup, we compare the switching union-CBFs designed in this paper with previously proposed BNCBF \cite{glotfelter2020nonsmooth} and composed CBF \cite{molnar2023composing}. The simulated trajectories are shown in Fig. \ref{fig:trajectory_comparison}.
The overall safety performance of switching union-CBFs, BNCBF, and compsed CBF are similar. However, the trajectories starting from the red dot in the middle left of Fig.\ref{fig:trajectory_comparison} show that the composed CBF may get stuck at some states due to unexpected equilibria while the BNCBF may get stuck due to over constraining on BNCBF-QP. In this single integrator example, changing the class-$\kappa$ function of composed CBF or BNCBF constraint only changes the distance between the sticking point and the obstacle. However, using the switching union-CBFs, since we only follow one CBF constraints at a time, the trajectories starting from the same initial states do not have stikcing points and sucessfully reach the goal point.

\subsection{Performance Comparison of Verif-I and II} 

This subsection compares the efficiency between Verif-I and Verif-II. We consider the polynomial control system from \cite{wang2024safe}.
\begin{equation*}
\label{eq:polynomial_system}
    \begin{bmatrix}
        \dot{x}_1\\
        \dot{x}_2
    \end{bmatrix} = 
    \begin{bmatrix}
        x_2\\
        x_1 + \frac{1}{3}x_1^3 + x_2
    \end{bmatrix} + 
    \begin{bmatrix}
        (0.2x_1^2 + 0.2x_2 + 1)u_1\\
        (-0.2x_2^2 + 0.2x_1+4)u_2
    \end{bmatrix}
\end{equation*}
with control limits $-5\leq u_1 \leq 5$ and $-5\leq u_2 \leq 5$. We first show that Verif-I is sensitive to CBFs' layout of the union, while Verif-II is not. Then, we show that the total verification time of Verif-I grows rapidly as the number of CBFs increases, while the verification time of Verif-II grows linearly. 

We first focus on relative positions of CBFs' super level sets, which is called \textit{CBF layout}. Depending on the number of non-empty intersections between super level sets, we consider two types of CBF layouts, namely sparse layout and dense layout, shown in Fig. \ref{fig:different_layouts}. In both layout cases, each CBF is $h_i(\mathbf{x}) = 0.04 - ||\mathbf{x}-\mathbf{x}_i||^2$, $i=1,2,3$. In the sparse case $\mathbf{x}_1 = [-0.15, 0.15]$, $\mathbf{x}_2 = [0, 0]$, and $\mathbf{x}_3 = [0.15, -0.15]$. In the dense case $\mathbf{x}_1 = [-0.1, 0]$, $\mathbf{x}_2 = [0.1, 0]$, and $\mathbf{x}_3 = [0, 0.17]$. In the sparse case shown in Fig. \eqref{fig:sparse_layout}, Verif-I takes 76.12 seconds and Verif-II takes 1.12 seconds. In the dense case shown in Fig. \eqref{fig:dense_layout}, Verif-I takes 463.10 seconds while Verif-II only takes 1.14 seconds. This time comparison shows that Verif-I is sensitive to the layout of the CBFs in the union. According to Section \ref{sec:verification_I}, the total number of non-empty subsets computed in the step 2 of Verif-I depends on the layout of CBFs, and it directly decides the number of times we repeat the Program \eqref{eq:T3_verification_prog}. On the other hand, Verif-II only repeats the Program \eqref{eq:verify_theorem4} for each of the CBFs in the union, thus not sensitive to the layout.

The next comparison shows that, in the sparse layout, for the same number of CBFs, Verif-I takes longer time than Verif-II. The result is shown in Fig. \ref{fig:computation_time_comparison}. Compared to Verif-II, the computation time of Verif-I grow rapidly as the number of CBFs grows. This is caused by the increasing number of indeterminate variables in the Program \eqref{eq:T3_verification_prog} when verifying the non-empty intersections between CBFs. On the contrary, Verif-II only repeats the Program \eqref{eq:thm4_verif_sos} for each of the 0-super-level set, and the amount of indeterminate variables in Program \eqref{eq:verify_theorem4} is the same for all CBFs in the union. Hence, the total computation time for Verif-II grows linearly with respect to the number of CBFs. 

The efficiency advantage of Verif-II over Verif-I is obvious. however, we also need to point that, for some cases, Verif-I provides valid verification results while Verif-II does not. Consider a 2D single integrator $\dot{\mathbf{x}} = [v, u]^T$ where $v$ is a constant velocity and $u\in[-1.5, 1.5]$ is the control input. This system is impossible to be forward invariant within any circular region. Hence, If we consider the union of CBF $h_1(\mathbf{x}) = 8 - (\mathbf{x}_{(0)} + 2)^2 - (\mathbf{x}_{(1)} + 2)^2$ and CBF $h_2(\mathbf{x}) = \mathbf{x}_{(0)} - \mathbf{x}_{(1)}$, the Verif-II should fail. But Verif-I could verifies this union since, starting from any point inside the circle $h_1(\mathbf{x})\geq 0$, we can always use switching strategy I to switch from $h_1(\mathbf{x})$ to $h_2(\mathbf{x})$. Fig \ref{fig:verif-II_invalid} shows the simulation result of this example.


\begin{figure}[h]
\centering
\begin{subfigure}{0.23\textwidth}
    \includegraphics[width=\textwidth]{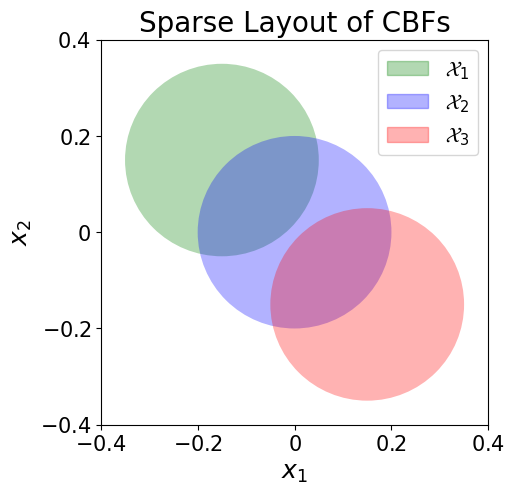}
    \subcaption{}
    \label{fig:sparse_layout}
\end{subfigure}
\begin{subfigure}{0.23\textwidth}
    \includegraphics[width=\textwidth]{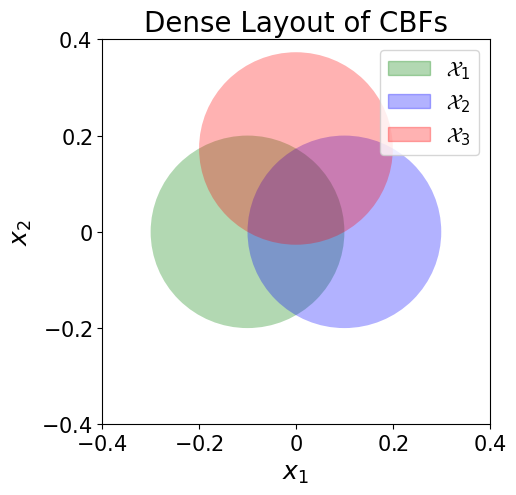}
    \subcaption{}
    \label{fig:dense_layout}
\end{subfigure}
\caption{Union of 3 CBFs with different layouts. Fig. \eqref{fig:sparse_layout} is the spares layout such that $\mathcal{X}_c$ only has 5 subsets $\mathcal{X}_{\mathcal{N}}$. Fig. \eqref{fig:dense_layout} is the dense layout such that $\mathcal{X}_c$ has 7 subsets $\mathcal{X}_{\mathcal{N}}$.}
\label{fig:different_layouts}
\end{figure}

\begin{figure}[h]
\centering
\begin{subfigure}{0.23\textwidth}
    \includegraphics[width=\textwidth]{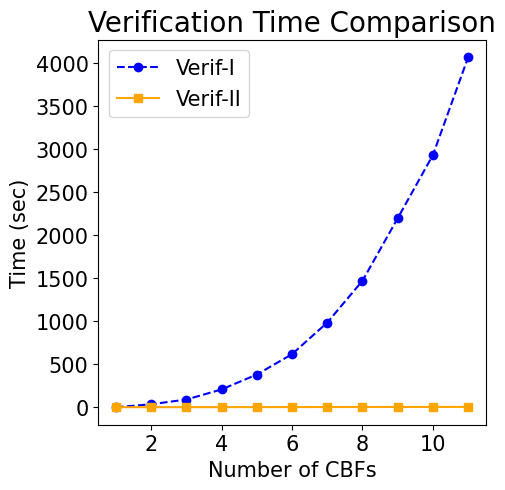}
    \subcaption{}
    \label{fig:computation_time}
\end{subfigure}
\begin{subfigure}{0.22\textwidth}
    \includegraphics[width=\textwidth]{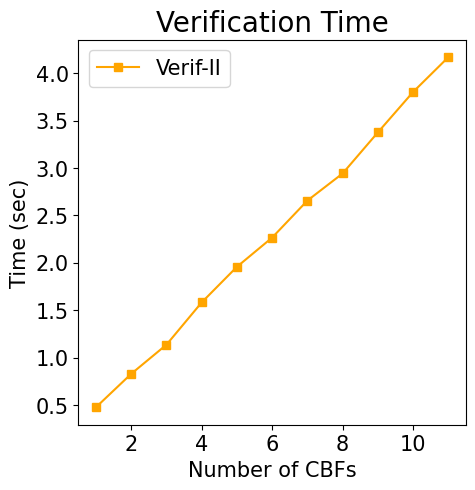}
    \subcaption{}
    \label{fig:computation_time_II}
\end{subfigure}
\caption{Computation time for Verif-I and Verif-II. Fig. \eqref{fig:computation_time} compares the computation time of Verif-I and Verif-II in the sparse case with unions of different number of CBFs. Fig. \eqref{fig:computation_time_II} shows that the computation time of Verif-II only grows linearly as the number of CBFs grows.}
\label{fig:computation_time_comparison}
\end{figure}

\begin{figure}[h]
    \centering
    \includegraphics[width=0.65\linewidth]{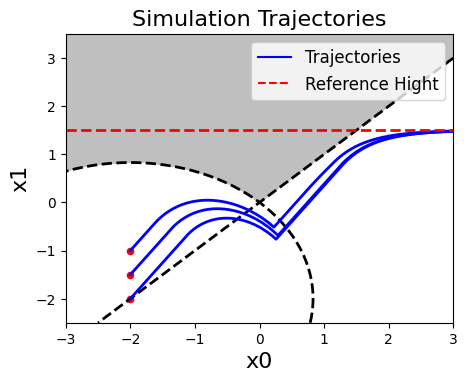}
    \caption{Simulated trajectories of single integrator with constant horizontal velocity. The boundaries of the CBFs are marked as black dashed lines. The CBF-QP uses switching strategy I to switch from circular CBF $h_1(\mathbf{x})$ to linear CBF $h_2(\mathbf{x})$. The closed loop system keeps safe and tracks the reference height.}
    \label{fig:verif-II_invalid}
\end{figure}
\section{Conclusions And Future Work}
\label{sec:conclusion}
This paper focuses on ensuring forward invariance of the union of super level sets of CBFs (union-CBFs). We proposed sufficient conditions for keeping the closed-loop system with switching CBF-QP controller safe. The proposed condition ensures (i) finite switches in any finite time interval and (ii) forward invariance of the closed-loop system. By utilizing the concept of minimal CBF and Slater's condition, our proposed sufficient condition allows less conservative safe control policies while satisfying the point-Lipschitz property. We considered two types of switching strategies and proposed sufficient conditions on CBFs for each of the strategies. We also derived SOS verification frameworks for union-CBFs conditions under each switching strategy. The experiments show that our switching union-CBFs avoids numerical issue brought about by the over constraining. However, the methods proposed in this paper only consider CBFs with relative degree 1, and hence our future work will extend the discussion to high-relative-degree cases.

\bibliographystyle{ieeetr}
\bibliography{ref}
\end{document}